\documentclass[11pt]{amsart}
\usepackage[left=0.55in,right=0.55in,top=0.59in,bottom=0.55in]{geometry}

\usepackage{amscd}
\usepackage{latexsym}
\usepackage[usenames]{color}
\usepackage{mathrsfs}
\usepackage{amsmath,amsfonts,amssymb,amsthm}
\makeindex

\newtheorem{Theorem}{Theorem}[section]
\newtheorem{Lemma}[Theorem]{Lemma}
\newtheorem{Notation}[Theorem]{Notation}
\newtheorem{Corollary}[Theorem]{Corollary}

\newtheorem{Proposition}[Theorem]{Proposition}
\theoremstyle{definition}

\newtheorem{Definition}[Theorem]{Definition}

\newtheorem{Examples}[Theorem]{Examples}

\newtheorem{Questions}[Theorem]{Questions}

\newtheorem{Remark}[Theorem]{Remark}
\newtheorem{Remarks}[Theorem]{Remarks}
\numberwithin{equation}{section}

\newcommand{\T}{\mathbb{T}}
\newcommand{\Z}{\mathbb{Z}}

\newcommand{\N}{\mathbb{N}}

\newcommand{\R}{\mathbb{R}}

\newcommand{\C}{\mathbf{c}}

\renewcommand{\theta}{\vartheta}
\newcommand{\eps}{\varepsilon}
\renewcommand{\phi}{\varphi}

\newcommand{\dis}{\displaystyle}

\newcommand{\cal}{\mathcal}

\newcommand{\wt}[1]{\widetilde{#1}}
\newcommand{\ol}[1]{\overline{#1}}

\def\SS_{{\mathfrak S}_{qc}}

\begin{document}                      

\title{On the Nuclearity of Dual Groups }

\author{Lydia Aussenhofer}
\address{Universit\"at Passau, Fakult\"at f\"ur Informatik und Mathematik, Innstr. 33, D-94032 Passau}
\email{lydia.aussenhofer@uni-passau.de}

\subjclass[2000]{22Axx, 46A11 }

\keywords{character group, dual space, $k_\omega$ space, nuclear group, nuclear vector space,  Pontryagin reflexive group, strongly reflexive group, 
 $s$--numbers}

\maketitle

\begin{abstract}
We prove that the dual space of a locally convex nuclear $k_\omega$
vector space  endowed with the compact--open topology is a locally
convex nuclear vector space. An analogous  result is shown for nuclear
groups. As a consequence of this, we obtain that   nuclear
$k_\omega$--groups are strongly reflexive.
\end{abstract}

\section{Introduction}
Since Grothendieck defined nuclear vector spaces in the fifties
(\cite{Grothendieck}), their properties, both in the context of the
theory of vector spaces and  their applications in  analysis  have
been studied.

In the monograph \cite{Pietschnuc} a survey is given about locally
convex nuclear vector spaces. Several years later, in the nineties,
Banaszczyk generalized the setting and introduced nuclear groups in
\cite{Buch}. He was able to show that nuclear groups share many
properties of locally convex nuclear  vector spaces and he was able
to prove a version of the Bochner and the L\'{e}vy Theorem for
nuclear groups. Even more, studying subgroups of locally convex
nuclear vector spaces (which are  typical examples for nuclear
groups) he gave surprising characterizations of nuclear Fr\'{e}chet
spaces by means of their subgroups (\cite{W+M}).

Nuclear groups have also good properties with respect to Pontryagin
duality. Before recalling them, we remember the fundamental
settings.

For an abelian topological group $G$, the set of continuous {\bf
characters} (i.e. homomorphisms in the compact circle group
$\T:=\{z\in\C|\ |z|=1\}$) forms a group with multiplication defined
pointwise. This group is denoted by $G^\wedge$ and called {\bf
character group} or {\bf dual group} of $G$. If $G^\wedge$ is
endowed with the compact--open topology, it is an abelian Hausdorff
group, in particular a topological group. This permits to form the
second dual group $G^{\wedge\wedge}:=(G^\wedge)^\wedge$. A
(necessarily abelian Hausdorff) topological group $G$ is called {\bf
Pontryagin reflexive} if the canonical mapping
$$\alpha_G:G\rightarrow G^{\wedge\wedge},\ x\mapsto (\chi\mapsto
\chi(x))$$ is a topological isomorphism.

These definitions imitate  the dual space of a topological vector
space and their duality theory. Fortunately, the two concept
coincide, which means:

\begin{Theorem}\label{Smith}
If $E$ is a topological vector space, the
mapping
$$E'_{co}\rightarrow E^\wedge,\ \phi\mapsto e^{2\pi i\phi}$$ from the
topological dual $ E'$   endowed with the compact--open topology
into $E^\wedge$ is a topological group isomorphism.
\end{Theorem}

A proof can be found in \cite{Smith}.

In the realm of topological vector spaces is an intensively studied object
(recall that the strong topology on
$E'$, a neighborhood basis of $0$ is given by the polars of bounded
subsets of $E$). In
\cite{Pietschnuc} those nuclear space are characterized the strong dual of which is
again a locally convex nuclear vector space.

Since we are interested in Pontryagin duality, the natural question
which arises is to find  sufficient conditions on  a nuclear space
$E$ (resp. a nuclear group $G$) such that $E´_{co}$  (resp.
$G^\wedge$) is a nuclear vector space.

One result in this direction in the realm of nuclear groups was
achieved in \cite{Buch} (16.1) and generalized in \cite{Diss}
(20.35) and (20.36):

\begin{Theorem}\label{20.36}
If $G$ is a metrizable nuclear group then $G^\wedge$ is a nuclear
group.
\end{Theorem}

From this it is easy to obtain

\begin{Corollary}
If $E$ is a metrizable nuclear locally convex vector space then
$E'_{co}$ is a locally convex nuclear space.
\end{Corollary}

One only has to consider  the isomorphism $E'_{co}\rightarrow
E^\wedge$  and to recall that a topological vector space is a
nuclear group if and only if it is a locally convex nuclear vector
space ((8.9) in \cite{Buch}).

The character group of a metrizable group is always a
$k_\omega$--space ((4.7) in \cite{Diss} or Theorem 1 in \cite{MJ})
which means it has a countable cobasis for the compact sets and the
topology is the final topology induced by the compact subsets.

Now it is natural to ask whether the dual group of a nuclear
$k_\omega$--group is a nuclear group again. In \ref{MainTheorem} we
will give an affirmative answer to this question. As a consequence
of this it is not difficult to deduce that a nuclear
$k_\omega$--group $G$ is {\bf strongly reflexive} which means that
all closed subgroups and all Hausdorff quotient groups of $G$ and
$G^\wedge$ are Pontryagin reflexive.

The article is organized as follows: In the second section we recall
the definition of $s$--numbers. A typical example are the
Kolmogoroff numbers which assign to two symmetric and convex subsets
$X$ and $Y$ of a real vector space $E$ a decreasing sequence
$(d_k(X,Y))$ of real numbers (or $\infty$). These sequences
measure  how big $X$ is w.r.t. $Y$. Drawing on results of
\cite{Bauhardt}, we show that $d_k(X,Y)$  can be  approximated
by $d_k(N\cap X,Y)$ for a suitable $k$--dimensional subspace $N$. This means:
If $X$ is rather big w.r.t. $Y$ then already a finite--dimensional
subset of $X$ is rather big.

Based upon  this result, we show in section 3 that a non--nuclear
real Frech\'{e}t space  has a compact convex set $K$ and a convex
$0$ neighborhood $U$ such that $d_k(K,U)$ is rather big, more
precisely $(k^3d_k(K,U))$ is unbounded.

In section 4 we define nuclear vector groups (locally convex vector
groups), which are roughly speaking locally convex nuclear vector
spaces (locally convex vector spaces) over $\R$ with the discrete
topology. We give a similar characterization as above for metrizable
locally convex vector groups  which are not nuclear vector groups.

Afterwards, in section 5, we recall both the definition and a
representation of nuclear groups. The latter allows us to find a
null--sequence $(g_n)$ and a neighborhood $U$ in a non--nuclear
metrizable and locally quasi--convex group $G$ such that $(d_k(\{\pm
g_n|\ n\ge n_0\}\cup\{0\},U))\le (ck^9)$ is impossible for every
$n_0\in \N$ if the constant $c$ is sufficiently small.

After these preparations, in section 6 we will first prove (based on
the results of section 3) that the dual space $E'_{co}$ of a locally
convex nuclear vector space $E$ which is a $k_\omega$--space, is
again nuclear. For groups, an analogous result is true, however, the
proof is more intricate and one more technical lemma is needed (in
order to manipulate the constant $c$ mentioned above).

Combining this result with well known results on nuclear groups and
$k_\omega$--groups, we obtain that a nuclear $k_\omega$--group is
strongly reflexive.

\section{A property of   Kolmogoroff numbers}

Let ${\cal L}$ denote the class of all bounded linear operators
between seminormed spaces. For seminormed spaces $E$ and $F$, the
set of bounded linear operators $E\rightarrow F$ will be denoted by
${\cal L}(E,F)$. We recall the following definition:

\begin{Definition}
A mapping $s$ assigning to every ${\cal L}\ni T:E\rightarrow F $  a
sequence of real numbers $(s_k(T))_{k\in\N}$  satisfying the
  properties (S 1) to (S 5) is called an {\bf $s$--number}.

\begin{itemize}
\item[(S1)] $\|T\|=s_1(T)\ge s_2(T)\ge\ldots \ge 0$,
\item[(S2)] $s_k(S+T)\le s_k(S) + \|T\|$ for all $S,T\in {\cal
L}(E,F)$ and for all $k\in\N$,
\item[(S3)] $s_k(R\circ S\circ T)\le \|R\|s_k(S)\|T\|$ for all
$T\in{\cal L}(E_0,E),\ S\in{\cal L}(E,F),\ R\in{\cal L}(F,F_0)$ and
$n\in\N$,
\item[(S4)] If ${\rm dim}\,T(E)<k$ then $s_k(T)=0$,
\item[(S5)] $s_k(I_k)=1$ for all $k\in\N$ where $I_k:\R^k\rightarrow
\R^k$ is the identity and $\R^k$ is endowed with the euclidean norm.
\end{itemize}
\end{Definition}

\begin{Remark}
Originally, $s$--numbers have been introduced in realm of bounded
linear operators between Banach spaces (\cite{Pietsch}).
\end{Remark}

Before presenting some examples, we introduce the following

\begin{Notation}
For a subspace $F$ of a seminormed space $E$,
 $$ Q^E_F:E\rightarrow E/F$$
denotes the canonical projection and $$ U_E=\{x\in E|\ \|x\|\le
1\}$$ the unit ball of $E$.
\end{Notation}

\begin{Examples}
Let $T:E\rightarrow F$ be a bounded linear operator between
seminormed   spaces.

\begin{itemize}
\item[(i)] The {\bf Kolmogoroff numbers} $(d_k)$,  defined by
  $$
  \begin{array}{lrl}
d_k(T)&:=&\inf \{\| Q^F_L\circ T\|\;|\ {\rm dim}L<k\} \\ [1ex]
       &= &\inf\{\inf\{c>0|\ T(U_E)\subseteq c U_F + L\}|\ {\rm dim }\,L<k\} \\[1ex]
       &= &\inf\{c>0|\  \exists L\le E,\ {\rm dim}L<k, U_E\subseteq cU_F+L \},
  \end{array}$$
are $s$--numbers.  {\rm [Cf. 11.6 in \cite{Pietsch}.]}

  More generally, for symmetric and convex subsets $X,Y$ in a vector
  space $E$, we put
  $$d_k(X,Y):=\inf\{c>0|\ \exists L\le E,\ {\rm dim}L<k,\ X\subseteq
  cY+L\}\in[0,\infty]$$
  Then $d_k(T)=d_k(T(U_E),U_F)$.
\item[(ii)] The {\bf Hilbert numbers} $(h_k)$, defined by $$
\begin{array}{ll }
h_k(T):= &\sup\{d_k(S_2\circ T\circ S_1)|\  \ S_1\in{\cal
L}(H_1,E),\ S_2\in{\cal L}(F,H_2),\\ [1ex]
 &\ \|S_1\|\le 1,\ \|S_2\|\le 1, \
H_1,H_2\mbox{   Hilbert spaces}\},\end{array}$$ are $s$--numbers.
{\rm [Cf. 11.4 in \cite{Pietsch}. The definition given in (11.4.1)
in  \cite{Pietsch} of Hilbert numbers is different. However, the
definitions coincide according to (11.3.4) in \cite{Pietsch}.]}

\end{itemize}
\end{Examples}

The following Remark will clarify the connection between the two
different settings of Kolmogoroff diameters.
\begin{Remark}\label{Remark}
Let $X_1$ and $X_2$ be symmetric and convex subsets of a vector
space $E$ with $d_1(X_1,X_2)<\infty$. Let $p_i$ denote the Minkowski
functional of $X_i$ defined on $\langle X_i\rangle$. Then
$$d_k(X_1,X_2)=d_k(T)\quad\mbox{where}\quad T:(\langle X_1\rangle,p_1)\rightarrow
(\langle X_2\rangle,p_2),\ x\mapsto x.$$

{\rm Indeed, for every $0<\eps<1$ we have $(1-\eps)B_{p_i}\subseteq
X_i\subseteq B_{p_i}$ and hence
$$(1-\eps)d_k(B_{p_1},B_{p_2})=d_k((1-\eps)B_{p_1},B_{p_2})\le$$ $$\le
d_k(X_1,X_2)\le d_k(B_{p_1},(1-\eps)B_{p_2})\le \frac{1}{1-\eps}
d_k(B_{p_1}, B_{p_2}).$$}
\end{Remark}

We gather some simple properties of the Kolmogoroff numbers, which
will be used frequently in the sequel.

\begin{Lemma}\label{properties}
\begin{itemize}
\item[(i)] $d_{k+l-1}(X,Z)\le d_k(X,Y)\cdot d_l(Y,Z) $  for all
$k,l\in\N$.
\item[(ii)] If $N$ is a subspace contained in $X$ and if
$d_1(X,Y)<\infty$, then $d_k(X,Y)=d_k(Q^E_N (X),Q^E_N(Y))$.
\item[(iii)] If $N$ is a subspace contained in $Y$, then
$d_k(X,Y)=d_k(X+N,Y)$.
\item[(iv)] For $\lambda,\mu>0$ we have
$d_k(\lambda X,\mu Y)=\frac{\lambda}{\mu}d_k(X,Y)$.
\item[(v)] $d_k(T^{-1}(T(X)),T^{-1}(T(Y)))=d_k(T(X),T(Y))$ for a
linear operator $T:E\rightarrow F$ and symmetric and convex subsets
$X,Y$ of $E$.
\end{itemize}
\end{Lemma}

\begin{Definition}
We define $$l_k(T):=\sup\{d_k(T|_{N})|\ N\le E,\ {\rm dim}N\le
k\}.$$
\end{Definition}
Then $(l_k)$ has  the properties (S2) to (S5) but it may fail to be
monotone.
 Nevertheless, the following estimate holds:

\begin{Theorem}\label{estimate}
$$l_k(T)\le d_k(T)\quad \mbox{and}\quad h_k(T)\le d_k(T)\le k^2 l_k(T).$$
\end{Theorem}

For the proof we need the following

\begin{Lemma}[Bauhardt]
\label{Bauhardt}

Let     $T\in {\cal L}(E, F)$.
\begin{itemize}
\item[(i)] For   $x_1,\ldots,x_k\in  U_E $ and
 $L_i:=\langle T(x_j):\ j\not=i\rangle $ ($i=1,\ldots, k$) the
 following estimate holds:
$$h_k(T)\ge \frac{1}{k}\min\{\|Q^F_{L_j}T(x_j)\|\ |\ \ j=1,\ldots,k\}$$
\item[(ii)] For every $0<\eps < 1$ and every $k\in\N$ there exist
$x_1,\ldots, x_k\in U_E$ such that $$\min\{\|Q^F_{L_j}T(x_j)\|\ |\
j=1,\ldots, k\}\ge \frac{1}{k}d_k(T)(1-\eps)^2.$$
\end{itemize}
\end{Lemma}

\begin{proof}
These are  Lemma 1 and Lemma 3 in \cite{Bauhardt}.
\end{proof}

\noindent{\it Proof} {\it of \ref{estimate}}.
The first inequality of \ref{estimate} follows from (S 3), since
$T|_N=T\circ \iota_N$ where $\iota_N:N\rightarrow E$ is the
embedding.

The inequality $h_k(T)\le d_k(T)$ is also a consequence of (S 3).

In order to prove the last inequality, we fix $0<\eps <1$ and
$k\in\N$. According to \ref{Bauhardt} (ii), there exist
$x_1,\ldots,x_k\in  U_E $ such that $$\min\{\|Q^F_{L_j}T(x_j)\|\ |\
j=1,\ldots, k\} \ge \frac{(1-\eps)^2}{k} d_k(T)$$ and \ref{Bauhardt}
(i) implies

$$h_k(T|_{\langle
x_1,\ldots,x_k\rangle})\ge \frac{1}{k} \min \{\|Q^{\langle
T(x_1),\ldots,T(x_k)\rangle}_{L_j}T(x_j)\|\ |\ \ j=1,\ldots, k\}.$$

Taking into consideration $\|Q^{\langle
T(x_1),\ldots,T(x_k)\rangle}_{L_j}T(x_j)\|=\|Q^F_{L_j}T(x_j)\|$ and
recalling that $d_k\ge h_k$,   we obtain:
$$d_k(T|_{\langle
x_1,\ldots,x_k\rangle})\ge h_k(T|_{\langle x_1,\ldots,x_k\rangle})
\ge \frac{(1-\eps)^2}{k^2}d_k(T).$$ Since $\eps$ was arbitrary, the
assertion follows. 

\begin{Corollary}\label{Cor}
Let $X$ and $Y$ be symmetric and convex subsets of a vector space
$E$ with $d_1(X,Y)<\infty$. Then $$l_k(X,Y):=\sup\{d_k(X\cap N,Y)|\
N\le E,\ {\rm dim}N\le k\}$$ satisfies $$
 l_k(X,Y)\le d_k(X,Y)\le k^2 l_k(X,Y).$$
\end{Corollary}

\begin{proof}
This follows immediately from \ref{estimate} and \ref{Remark}.
\end{proof}

\section{A characterization of non--nuclear metrizable locally convex spaces}

In this section we define  locally convex nuclear vector spaces and
prove a characterization of non--nuclear metrizable locally convex
spaces.

 \begin{Definition}
A (real or complex) locally convex vector space   $E$ is called a
{\bf locally convex  nuclear vector space}   if for every symmetric
and convex neighborhood $U$ of $0$ there exists a symmetric and
convex neighborhood $W$ of $0$ such that $d_k(W,U)\le 1/k$ for all
$k\in\N$.

\end{Definition}

\begin{Remark}\label{Bemerkung}
The statement
$$d_k(W,U)\le 1/k$$ in the definition of a  locally convex nuclear vector
 can be replaced by $$d_k(W,U)\le k^{-\eps}$$
for any $\eps>0$. {\rm [This is an easy consequence of
\ref{properties} (i).]}
\end{Remark}

\begin{Proposition}\label{3.3}
 Let $E$ be a
locally convex vector space.  The following assertions are
equivalent:
\begin{itemize}
\item[(i)]
$E$ is a locally convex  nuclear vector space,
\item[(ii)]
  for every symmetric and convex $0$--neighborhood $U$ there is
a symmetric and convex $0$--neighborhood $W$ such that
$(kd_k(W,U))_{k\in\N}\in \ell^\infty$,
\item[(iii)]
for every $m\in\N$ and
  for every symmetric and convex $0$--neighborhood $U$ there is
a symmetric and convex $0$--neighborhood $W$ such that\\
$(k^md_k(W,U))_{k\in\N}\in \ell^\infty$.
\end{itemize}
\end{Proposition}

\begin{proof}
This is a direct consequence of \ref{Bemerkung} and \ref{properties}
(iii).\end{proof}

\begin{Theorem}\label{CharFrechet}
Let $E$ be a metrizable locally convex vector space which is not
nuclear. Then there exists a symmetric and convex, totally bounded
subset $K$ in $E$ and a symmetric and convex neighborhood $U$ of
$0$ such that $(k^3d_k(K,U))$ is unbounded.

Conversely, if $E$ is a  locally convex  nuclear vector space,
$K\subseteq E$ is a totally bounded symmetric and convex subset, $U$
is an arbitrary neighborhood of $0$, and $n\in\N$, then the
sequence $( k^nd_k(K,U))$ is bounded.

\end{Theorem}

\begin{proof}
Let us assume first that $E$ is a metrizable locally convex  nuclear
vector space which is not nuclear. According to \ref{3.3}, there
exists a symmetric and convex neighborhood $U$ of $0$ such that for
any other symmetric and convex neighborhood $W$ of $0$
$$(kd_k(W,U))\notin\ell^\infty\quad(\ast).$$

Let $(U_n)$ be a decreasing neighborhood basis of $0$ consisting of
symmetric and convex sets. We may assume  that $U_n\subseteq U$ for
all $n\in\N$. $(\ast)$ implies that for every $n\in\N$ there exists
$k_n\in \N$ such that $d_{k_n}(U_n,U)> \frac{n}{k_n}$.

According to \ref{Cor}, $\dis l_{k_n}(U_n,U)\ge \frac{1}{k_n^2}
d_{k_n}(U_n,U)> \frac{n}{k_n^3}$ and hence
 the following holds:
$$\forall n\in\N\  \exists N_n\le E,\ {\rm
dim}(N_n)\le k_n: \ d_{k_n}(U_n\cap N_n,U)\ge
 \frac{n}{k_n^3}.$$

The maximal subspace contained in the symmetric and convex set
$U_n\cap N_n$ is $\{0\}$. Indeed, suppose this subspace $L$ is not
trivial; consider the canonical projection $Q^E_L:E\rightarrow E/L$.
Since $d_1(U_n\cap N_n,U)\le d_1(U_n,U)<\infty$, \ref{properties}
implies $\dis d_{k_n}(U_n\cap N_n,U)=d_{k_n}(Q^E_L(U_n\cap
N_n),Q^E_L(U))=0$. The last equality is a consequence of (S 4),
since ${\rm dim} Q^E_F(N_n)<k_n$.

For $K:={\rm conv}\bigcup_{n\in\N}(U_n\cap N_n)$ we obtain
$$d_{k_n}(K,U)\ge d_{k_n}(U_n\cap N_n,U)\ge\frac{n}{k_n^3}\quad\forall\, n\in\N.$$ Hence it
suffices to show that $K$ is totally bounded.

\vspace{0.3cm}

 The sets $U_n\cap N_n$ are bounded
 (since they contain only the trivial subspace)
and therefore  totally bounded, by the equivalence of norms on
finite--dimensional vector spaces.
  For a
fixed neighborhood $W$ of $0$ there exists $m\in \N$ such that
$U_m+U_m\subseteq W$. Since $\dis {\rm conv}
\bigcup_{n=1}^{m-1}U_n\cap N_n$  is totally bounded (cf. (5.1), p.25
and (4.3), p.50 in \cite{Schaefer}), there exists a finite set $F$
such that $\dis {\rm conv} \bigcup_{n=1}^{m-1}U_n\cap N_n \subseteq
F+U_m$. Since the $(U_n)$ are decreasing and $U_m$ is convex, $\dis
{\rm conv} \bigcup_{n\ge m}U_n\cap N_n\subseteq U_m$. Hence we
obtain:
$${\rm conv} \bigcup_{n\ge 1} U_n\cap N_n\subseteq {\rm conv}
\bigcup_{n=1}^{m-1}U_n\cap N_n\, +\, {\rm conv} \bigcup_{n\ge m}
U_n\cap N_n\subseteq F+U_m+U_m\subseteq F+W,$$ which shows that $K$
is totally bounded.

\vspace{0.3cm}

Conversely, assume that $E$ is a locally convex  nuclear vector
space, $K$ is a symmetric and convex totally bounded subset and $U$
is a symmetric and convex neighborhood of $0$. For $n\in \N$, there
exists a neighborhood $W$ such that $d_k(W,U)\le k^{-n}$
(\ref{Bemerkung}). Since $K$ is totally bounded, there exists a
finite, symmetric set $F$ such that $K\subseteq F+W$. Let $f:={\rm
dim }\langle F\rangle$. Then $d_{k+f}(K,U)\le d_{k+f}({\rm conv}F+W,
U)\le d_k(W,U)\le k^{-n}\le c\cdot (k+f)^{-n}$ for suitable $c>0$.
The assertion follows.

\end{proof}

\begin{Corollary}
For a   metrizable locally convex  vector space  $E$ the following
assertions are equivalent:
\begin{itemize}
\item[(i)] $E$ is nuclear.
\item[(ii)] For every symmetric and convex totally bounded subset
$K$, for every symmetric and convex $0$--neighborhood $U$ the
sequence $(k^3d_k(K,U))$ is bounded.
\end{itemize}
\end{Corollary}

\begin{Remark}
The vector space $\R^{(I)}$ endowed with the asterisk--topology (which makes $\R^{(I)}$ the locally convex direct sum)
 is
not nuclear if $I$ is uncountable. However, the totally bounded
subsets are contained in finite--dimensional subspaces.This implies
that for every totally bounded, symmetric and convex set $K$ and
every symmetric and convex $0$--neighborhood $U$ the sequence
$(d_k(K,U))$ is eventually zero, in particular $(k^nd_k(K,U))$ is
bounded for every $n\in\N$.
\end{Remark}

\section{A property of metrizable nuclear vector groups which are not nuclear}

In this section  we recall the definition of locally convex vector groups and  nuclear
vector groups, which are, roughly speaking, locally convex spaces,
respectively locally convex  nuclear vector spaces over the scalar
field $\R$ endowed with the discrete topology. These are of
interest, since the class of Hausdorff quotients of subgroups of
nuclear vector groups coincides  with the class of all nuclear
groups (to be defined in \ref{nuclear}). We show a property of
non--nuclear metrizable locally convex vector groups which will be
applied to prove an analogous property for non--nuclear metrizable
groups (\ref{anstrengend}). This will be the key lemma in order to
show that the character group of a nuclear $k_\omega$--group is
again nuclear.

 \begin{Definition}
 A vector space $E$ endowed with a Hausdorff group topology ${\cal
 O}$ is called a {\bf locally convex vector group} if there is a
 neighborhood basis of $0$ consisting of symmetric and convex sets.

A    locally convex vector group   $E$ is called a {\bf nuclear
vector group}  if for every symmetric and convex neighborhood $U$
of $0$ there exists a symmetric and convex neighborhood $W$ of $0$
such that $d_k(W,U)\le 1/k$ for all $k\in\N$.
\end{Definition}

\begin{Remark}\label{EigVG}
\begin{itemize}
\item[(i)] According to Proposition 3 in \cite{Kenderov}, $(\R^n,{\cal
O})$ is a locally convex vector group if and only if there exists a
basis $(b_1,\ldots,b_n)$ of $\R^n$ and $0\le m\le n$ such that
${\cal O}$ induces the usual topology on ${\langle
b_1,\ldots,b_m\rangle}_{\R}$ and the discrete topology on ${\langle
b_m,\ldots,b_n\rangle}_{\R}$.

This shows, that in a locally convex vector group $(E,{\cal O})$ a
set of the form $\{tx|\ |t|\le 1\}$ ($x\in E$) is not necessarily
compact.

\item[(ii)] However, the following holds:
 Let $E$ be a
locally convex vector group. For every $\lambda\in \R$, the scalar
multiplication $m_\lambda:E\rightarrow E,\ x\mapsto \lambda x$ is
continuous.

\item[(iii)]This property shows that analogous formulations of \ref{Bemerkung}
and \ref{3.3} hold true for nuclear vector groups instead of nuclear
vector spaces.
\end{itemize}

\end{Remark}

\begin{Lemma}\label{finite}
Let $Y$  be a bounded, symmetric and convex subset  of $\R^n$. For
every $\eps>0$, there exists a finite subset $F\subseteq Y$ such that
$(1-\eps)Y\subseteq {\rm conv}(F )$.
\end{Lemma}

\begin{proof}
Without loss of generality we may assume that $\langle
 Y\rangle=\R^n$ and that $Y$ is compact if $\R^n$ is endowed with the usual
 topology.

[The first assumption is trivial, the second follows from $\dis
Y\subseteq \ol{Y}\subseteq (1+\eps)Y$. So if there exists a finite
set $F\subseteq \ol{Y}$ such that $\dis (1-\eps)\ol{Y}\subseteq{\rm
conv}(F)$, it follows that $\dis
\frac{1-\eps}{1+\eps}\ol{Y}\subseteq (1+\eps)^{-1}{\rm conv F}={\rm
conv}((1+\eps)^{-1} F)$. Since $\dis 1/(1+\eps)F$ is a finite subset
of $Y$, it follows that $Y$ may be assumed to be compact.]

 Let ${\cal F}$ be the set of all finite, symmetric subsets of
$  Y$ which contain $n$ linearly independent vectors. This set is
not empty. For $F\in {\cal F},$ the set ${\rm conv}(F )$  has
non--empty interior and it is easy to check that $\dis
(1-\eps)Y\subseteq\bigcup_{F\in{\cal F}}{\rm int}({\rm conv}(F))$.
Since $(1-\eps)Y$ is compact, there exists a finite subcover and the
union of the finite sets appearing in the subcover has the desired
properties.
\end{proof}



 \begin{Theorem} \label{kompakt}
Let $E$ be a metrizable, non--nuclear locally convex vector group.

 There exists a symmetric and convex
neighborhood $U$ of $0$ and a null sequence $(y_n) $ in $E$ such
that for all $n_0\in\N$

 $$(k^3 d_k({\rm conv}(\{y_n|\ n\ge n_0\}) ,U))\notin\ell^\infty.$$
\end{Theorem}

\begin{proof}
We fix a decreasing neighborhood basis $(U_n)$ consisting of
symmetric and convex sets.
  Since $E$ is not nuclear, there exists, according to \ref{3.3} and \ref{EigVG} (iii), a
  symmetric and  convex
$0$--neighborhood $U$ such that
$$(kd_k(U_n,U))\notin\ell^\infty\quad \forall n\in\N.
$$
We may assume that $U_n\subseteq U$ for all $n\in\N $.

Hence we have:
\begin{eqnarray}\forall\; n\in\N\quad\exists\; k_n\in\N\quad k_n
d_{k_n}(U_n,U)>  n.\end{eqnarray} According to \ref{Cor}, $\dis
l_{k_n}(U_n,U)\ge \frac{1}{k_n^2}d_{k_n}(U_n,U)>\frac{n}{k_n^3},$ so
there exists a sequence $(E_n)$ of
  subspaces where ${\rm dim}E_n\le k_n$
such that
\begin{eqnarray}d_{k_n}( U_n\cap E_n,U) \ge
\frac{n}{k_n^3}\quad\quad \forall n\in\N.\end{eqnarray}

As in the proof of \ref{CharFrechet} one shows that $U_n\cap E_n$ is
bounded in the finite dimensional space $E_n$.

According to \ref{finite}, there exists a finite, symmetric set
$F_n\subseteq U_n\cap E_n$ such that ${\rm conv}F_n \supseteq
 \frac{1}{2} (U_n\cap E_n)$; hence we obtain
\begin{eqnarray}d_{k_n}( {\rm
conv}(F_n),U)\ge \frac{1}{2}d_{k_n}(U_n\cap
E_n,U)\stackrel{(2)}{\ge} \frac{n}{2k_n^3} \quad\quad \forall
n\in\N.\end{eqnarray}

Since $F_m\subseteq U_m$, there exists a null sequence
$(y_j)_{j\in\N}$ such that $\bigcup_{m\in\N} F_m=\{\pm y_j|\
j\in\N\}$ and which satisfies that for every $n\in \N$ there exists
$m_n\in\N$ such that $\bigcup_{m\ge m_n} F_m\subseteq\{\pm y_j|\
j\ge n\}$. Since
$$d_{k_m}({\rm conv}\{\pm y_j|\ j\ge n \},U)\ge
 d_{k_m}({\rm conv}(F_m),U)\stackrel{(3)}{\ge} \frac{m}{2k_m^3}\quad\forall\
m\ge m_n,$$  we obtain $(k^3d_k({\rm conv}\{y_j|\ j\ge
n\},U))\notin\ell^\infty$  for every $n\in\N$.
\end{proof}

\section{A property of non--nuclear metrizable groups}

\begin{Notation}\label{not}
For an abelian group $G$ and   subsets $A$ and $B$ of $G$ we put
$$(d_k(A,B))\le (c_k)$$ (for $(c_k)\in [0,\infty]^{\N}$) if there
exists a vector space $E$ and symmetric and convex subsets $X$ and
$Y$ with $d_k(X,Y)\le c_k$ (for all $k\in \N$),  a subgroup $H$ of
$E$, and a homomorphism $\phi:H\rightarrow G$ such that $A\subseteq
\phi(X\cap H)$ and $\phi(Y\cap H)\subseteq Y$.
\end{Notation}

\begin{Remark}\label{surjektiv}
Without loss of generality one may assume that the homomorphism
$\phi$ in \ref{not} is surjective. {\rm [We may replace $E$
by $E\times \R^{(G)}$, $H$ by $\Z^{(G)}$, $\phi$ by the surjective
homomorphism $H\times \Z{(G)}\rightarrow G,\ (h,(k_g))\mapsto h+\sum
k_g g$, and $X,\ Y$ by $X\times \{0\}$, $Y\times \{0\}$,
respectively.]}
\end{Remark}

\begin{Definition}\label{nuclear}
An abelian Hausdorff group $G$ is called a {\bf nuclear group} if
for every $m\in\N$, every $c>0$, and every neighborhood $U$ of the
neutral element $0$ there exists another neighborhood $W$ of $0$
such that
$$(d_k(W,U))\le (ck^{-m}).$$
\end{Definition}

\begin{Definition}
For a subset $A$ of a topological group $G$, the set
$A^\triangleright:=\{\chi\in G^\wedge|\ \chi(A)\subseteq \T_+\}$
(with $\T_+:=\{z\in\T|\ {\rm Re}z\ge 0\}$) is called {\bf polar} of
$A$. Conversely, for a subset $B\subseteq G^\wedge$, we define
$B^\triangleleft:=\{x\in G:\ \chi(x)\in\T_+\ \forall\chi\in G^\wedge
\}$.

 A subset $A$ of a topological
group $G$ is called {\bf quasi--convex} if
$A=(A^\triangleright)^\triangleleft$. A topological group $G$ is
called {\bf locally quasi--convex} if there is a neighborhood basis
of the neutral element consisting of quasi--convex sets.
\end{Definition}

\begin{Remarks}
The notation of  quasi--convexity is  a generalization of the
description of convex sets given by the Hahn--Banach theorem.

For an arbitrary subset $A\subseteq G$, the set
$(A^\triangleright)^\triangleleft=:{\rm qc}(A)$ is the smallest
quasi--convex subset which contains $A$. It is called the {\bf
quasi--convex hull} of $A$.

Every locally quasi--convex Hausdorff group is maximally almost
periodic (i.e. the continuous characters separate the points).

\end{Remarks}

\begin{Theorem}
Every nuclear group is locally quasi--convex.
\end{Theorem}

\begin{proof}
This is Theorem (8.5) in \cite{Buch}.
\end{proof}

 Now we repeat and generalize a representation of locally quasi--convex (nuclear)
 groups as   quotients of   subgroups of  locally convex (nuclear)
 vector groups given in (9.6) in \cite{Buch}, since we need a quantitative
 version here.

 Let $G$ be a locally quasi--convex
group. For  a nonempty  subset $A\subseteq G$ (not necessarily a
neighborhood), we put
$$X_A:=\{(x_\chi)\in\R^{G^\wedge}|\ \exists g\in A\ \mbox{such
that}\ e^{2\pi i x_\chi}=\chi(g)\ \mbox{and}\
  |x_\chi|\le 1/4\   \forall\chi\in A^\triangleright\},$$
$$Y_A:={\rm conv}(X_A),$$ and
$$H_0:=\{(x_\chi)_{\chi\in G^\wedge}\in \R^{G^\wedge}
|\ e^{2\pi i x_\chi}=\chi(g)\ \mbox{for some } g\in G \ \mbox{and
all}\ \chi\in G^\wedge\}.
$$
 Since the characters separate the points, the element $g\in G$ is
unique and hence $$\phi_0:H_0\rightarrow G,\ (x_\chi)\mapsto g$$ (if
$\chi(g)=e^{2\pi i x_\chi}$ for all $\chi\in G^\wedge$) is well
defined. It is easy to prove that $\phi_0$ is a group homomorphism.

For quasi--convex $0$--neighborhoods $U$ and $W$ such that
$U+U\subseteq W$ it is proved in (9.6) in \cite{Buch} that
$X_U+X_U\subseteq X_W$ and hence  $$ X_U\cap H_0+X_U\cap
H_0\subseteq X_W\cap H_0,\ $$ $$ {\rm conv}( X_U \cap H_0) +{\rm
conv}( X_U\cap H_0)\subseteq {\rm conv}(X_W\cap H_0),\ \mbox{and}\
Y_U+Y_U\subseteq Y_W.$$  This shows that $(Y_U)$ and $({\rm
conv}(X_U\cap H_0))$ (where $U$ runs through all quasi--convex
neighborhoods of $0$ in $G$) form neighborhood bases of locally
convex vector group topologies ${\cal O}$ and ${\cal O}_{co}$ on
$\R^{G^\wedge}$. (It is not difficult to verify that they are
Hausdorff spaces.) Instead of $(\R^{G^\wedge},{\cal O}_{co})$ we
simply write $V_0$.

 For a quasi--convex set $A$ we have \begin{eqnarray} Y_A\cap
H_0=X_A\cap H_0.\end{eqnarray} The inclusion "$\supseteq$" is
trivial. Conversely, fix $(x_\chi)\in Y_A\cap H_0$. Since
$(x_\chi)\in H_0$, there exists $g\in G$ such that $e^{2\pi i
x_\chi}=\chi(g)$ for all $\chi\in G^\wedge$. Since $(x_\chi)$
belongs to $Y_A\quad |x_\chi|\le 1/4$ for all $\chi\in
A^\triangleright$ and hence $\chi(g)\in\T_+$ for all $\chi\in
A^\triangleright$, which implies $g\in {\rm qc}(A)=A$ and hence
$(x_\chi)\in X_A\cap H_0$.


In     Lemma \ref{nucVG} we will   show that $G$ is a nuclear group
if and only if $V_0$ is a nuclear vector group, both endowed with
the topology induced by the neighborhood basis $(Y_U)$ and $({\rm
conv}(H_0\cap X_U))$. For its proof, we will need some lemmas.

 Let us first recall that a seminorm $p$ on a vector space $E$ is called {\bf
pre--Hilbert seminorm} if the parallelogram law
$p(x+y)^2+p(x-y)^2=2(p(x)^2+p(y)^2)$ holds for all $x,y\in E$.

\begin{Lemma}\label{2.14}
For every $m\ge 2$ there exists a constant $c_m>0$ such that for all
  symmetric and convex sets $X$ and $Y$  in a vector space which
  satisfy
  $d_k(X,Y)\le
ck^{-m}$ for some $m\ge 2$, there exist pre--Hilbert seminorms $p,q$
on $\langle X\rangle$ such that $X\subseteq B_p,\ \ B_q\subseteq Y$
and $d_k(B_p,B_q)\le c c_m k^{-m+2}$.
\end{Lemma}

\begin{proof}
This is Lemma (2.14) in \cite{Buch}. A proof can be found in
\cite{Diss}, (18.32) and (18.33).
\end{proof}

\begin{Lemma}\label{2.15}
Let $p$ and $q$ be pre--Hilbert seminorms on a vector space $E$ such
that $d_k(B_p,B_q)\rightarrow 0$ and $d_1(B_p,B_q)<\infty$. For
decreasing sequences of positive real numbers $(a_k)$ and $(b_k)$
with $d_k(B_p,B_q)\le a_kb_k$ there exists a pre--Hilbert seminorm
$r$ on $E$ such that $d_k(B_p,B_r)\le a_k$ and $d_k(B_r,B_q)\le b_k$
for all $k\in\N$.
\end{Lemma}

\begin{proof}
This is Lemma (2.15) in \cite{Buch}. A proof can be found in (18.28)
in \cite{Diss}.
\end{proof}

\begin{Lemma}\label{8.1}
Let $r$ and $q$ be pre--Hilbert seminorms on a vector space $E$
satisfying $\sum_{k\in\N}d_k(B_r,B_q)<\infty$ and let
$\chi:H\rightarrow \T$ be a homomorphism where $H$ is a subgroup of
$E$. If $\chi(B_q\cap H)\subseteq \T_+$ then there exists a linear
mapping $f:E\rightarrow \R$ such that $e^{2\pi i f}|_{H}=\chi$ and
$$\sup\{|f(x)|\ |\  x\in B_r\}\le
\frac{21}{2\pi}\sum_{k\in\N}d_k(B_r,B_q).$$
\end{Lemma}

\begin{proof}
This is (8.1) in \cite{Buch}. (The proof of the sharper estimate can
be found in (19.14) (ii) in \cite{Diss}.)
\end{proof}

\begin{Lemma}\label{3.20}
Let $p$ and $q$ be pre--Hilbert seminorms on a vector space $E$ with
$\sum_{k\in\N}d_k(B_p,B_q)^2<\frac{1}{4}$. Then for any subgroup $H$
of $E$ the following estimate holds:
$$d_k({\rm conv}(H\cap B_p),{\rm conv}(H\cap B_q))\le 2
d_k(B_p,B_q).$$
\end{Lemma}
\begin{proof}
This is Corollary (3.20) in \cite{Buch}.
\end{proof}

\begin{Lemma}\label{general}
Let $A$ be an arbitrary nonempty subset   and let $B\supseteq A$ be a
quasi--convex subset of an abelian topological group $G$   such that
 $(d_k(A,B))\le (ck^{-m})$. If $(8\pi c
c_m)^2\sum_{k\in\N}(k^{-m+4})^2<1/4$, then
$$d_k(Y_A,Y_B)\le 16\pi c
c_mk^{-m+4}.$$ If  $(16 \pi cc_m
c_{m-4})^2\sum_{k=1}^\infty (k^{- m+6})^2<1/4$ 
then
$$d_k({\rm conv}(X_A\cap H_0),{\rm conv}(X_B\cap H_0))\le 32\pi c c_m
c_{m-4}k^{-m+6}.$$
 \end{Lemma}

\begin{proof}
Observe that $m\ge 5$.

 By assumption, there is a vector space $E$, a
subgroup $H$ of $E$ and a homomorphism $\phi:H\rightarrow G$ and
there are symmetric and convex sets $X$ and $Y$ such that
\begin{eqnarray} d_k(X,Y)\le ck^{-m},\quad A\subseteq\phi(X\cap
H),\quad\mbox{and}\quad \phi(Y\cap H)\subseteq B. \end{eqnarray}

According to \ref{2.14}, there are pre--Hilbert seminorms $p$ and
$q$ defined on $\langle X\rangle $ such that \begin{eqnarray}
d_k(B_p,B_q)\le cc_m k^{-m+2}\quad\mbox{and}\quad  X\subseteq B_p,\
\ B_q\subseteq Y .\end{eqnarray}  As a consequence of \ref{2.15},
there exists a pre--Hilbert seminorm $r$ such that
\begin{eqnarray}d_k(B_p,B_r)\le 8 \pi  c c_m k^{-m+4}\quad\mbox{and}\quad
d_k(B_r,B_q)\le \frac{1}{8\pi}k^{-2}.\end{eqnarray}

Let us fix $\chi\in B^\triangleright$. Then $\chi\circ\phi(B_q\cap
H)\stackrel{(6)}{\subseteq} \chi(\phi(Y\cap
H))\stackrel{(5)}{\subseteq} \chi(B)\subseteq \T_+$. According to
\ref{8.1}, there is a linear function $\dis  f_\chi:\langle
X\rangle\rightarrow \R$ such that  $$  e^{2\pi i
f_\chi}|_{H}=\chi\circ \phi\quad\mbox{and}$$ \begin{eqnarray}
\sup\{|f_\chi(x)|\;|x\in B_r\}\le
\frac{21}{2\pi}\sum_{k\in\N}d_k(B_r,B_q)=
\frac{21}{2\pi}\cdot\frac{1}{8\pi}\cdot\frac{\pi^2}{6}<\frac{1}{4}.\end{eqnarray}

For $\chi\notin B^\triangleright$ we put $f_\chi\equiv 0$. We define
$$\Phi:\langle X\rangle\rightarrow V_0=\R^{G^\wedge},\ x\mapsto
(f_\chi(x))_{\chi\in G^\wedge}$$ and let $$P:\R^{G^\wedge}\rightarrow
\R^{B^\triangleright}\times \{0\}^{G^\wedge\setminus
B^\triangleright}$$ be  the canonical projection.

Then the following inclusions hold:

\begin{eqnarray}
 P(H_0)\cap \Phi(B_r)&\subseteq& P(X_B\cap H_0)\\
P(X_A)&\subseteq& P(H_0)\cap \Phi(B_p)
\end{eqnarray}

In order to prove (9), we fix $(x_\chi) \in P(H_0)\cap\Phi(B_r)$.
Since $(x_\chi)\in \Phi(B_r)$, there exists $x\in B_r$ such that
$x_\chi=f_\chi(x)$ for all $\chi\in B^\triangleright$. In
particular, $\dis e^{2\pi i x_\chi}=e^{2\pi i f_\chi(x)} $ for all
$\chi\in B^\triangleright$. On the other hand, $(x_\chi) \in
P(H_0)$, which means that there exists $g\in G$ such that
$\chi(g)=e^{2\pi i x_\chi}=e^{2\pi i f_\chi(x)}$ for all $\chi\in
  B^\triangleright$. Since
$x\in B_r$, $\dis |f_\chi(x)|\le 1/4$ (by (8)). This shows that
$\chi(g)\in \T_+$ for all $\chi\in B^\triangleright$ which implies
$g\in B$, since $B$ was assumed to be quasi--convex. It follows that
$(x_\chi)\in P(X_B \cap H_0)$.

Now we are going to prove (10).  Observe first that $A\subseteq B$
implies $B^\triangleright\subseteq A^\triangleright$. We fix
$(y_\chi)\in P(X_A)$. This means that $y_\chi=0$ for all $\chi\notin
B^\triangleright$ and $|y_\chi|\le 1/4$ for all $\chi\in
B^\triangleright\subseteq A^\triangleright$; further, there exists
$g\in A$ such that $\chi(g)=e^{2\pi i y_\chi}$ ($\forall\chi\in
B^\triangleright$). The latter equality implies immediately
$P((y_\chi))\in P(H_0).$

Since $A\subseteq \phi( H\cap B_p)$ (according to (5) and (6)),
there exists $x\in H\cap B_p$ with $g=\phi(x) $. We want to show
that $(y_\chi)=\Phi(x)=(f_\chi(x))$. Therefore we fix $\chi\in
B^\triangleright$. By assumption and (7), $\sum d_k(B_p,B_r)^2<1/4$,
in particular $d_1(B_p,B_r)<1$, which implies $B_p\subseteq B_r$. As
a consequence of this, $|f_\chi(x)|\le 1/4$ for all $\chi\in
B^\triangleright$ (by (8)). The assertion follows from
 $  e^{2\pi i f_\chi(x)}=\chi(\phi(x))=\chi(g)=e^{2\pi i y_\chi}$
and the fact that $|y_\chi|\le 1/4$ and $|f_\chi(x)|\le 1/4$. For
$\chi\notin B^\triangleright$, $y_\chi=f_\chi(x)=0$.

\begin{eqnarray}
Y_A\subseteq P^{-1}(P(Y_A))\quad\mbox{and}\quad Y_B=P^{-1}(P(Y_B))
\end{eqnarray}

The inclusion is trivial while the equality is a consequence of the
fact that $Y_B+{\rm ker}P=Y_B$.

We obtain:
 \begin{eqnarray}
d_k(Y_A,Y_B)&\stackrel{(11)}{\le} & d_k(P^{-1}(P(Y_A)),P^{-1}(P(Y_B)))\nonumber\\
&\stackrel{\ref{properties} (v)}{=}&d_k(P(Y_A),P(Y_B))\nonumber\\
&=&d_k({\rm conv}(P(X_A)),{\rm conv}(P(X_B)))\nonumber\\
&\stackrel{(9),(10)}{\le}&d_k({\rm conv}(P(H_0)\cap \Phi(B_p)),{\rm
conv}(P(H_0)\cap \Phi(B_r)))\end{eqnarray}

By assumption and  (7), $\dis
\sum_{k\in\N}d_k(B_p,B_r)^2<\frac{1}{4}$, hence \ref{3.20} implies
\begin{eqnarray}
 d_k({\rm conv}(P(H_0)\cap \Phi(B_p)),{\rm
 conv}(P(H_0)\cap \Phi(B_r))) \le \nonumber \\
 2d_k(\Phi(B_p),\Phi(B_r))\le 2d_k(B_p,B_r).
\end{eqnarray}

(Observe that the Minkowski functionals associated to $\Phi(B_p) $
and $\Phi(B_r)$ defined on $\Phi(\langle X\rangle$) satisfy the
parallelogram law.) Combining   (12), (13), and (7), we obtain:
\begin{eqnarray}
d_k(Y_A,Y_B)\le 16 \pi  c c_m k^{-m+4}.
\end{eqnarray}

Let us assume now that $ ( 16 \pi c c_m c_{m-4}
)^2\sum_{k=1}^{\infty}(k^{-m+6})^2<1/4$. This implies in particular
 $m\ge 7$. Hence there exist, according to \ref{2.14} and (14), pre--Hilbert
 seminorms $p'$ and $q'$ defined on $\langle Y_A\rangle$ such that
 $$Y_A\subseteq B_{p'},\ B_{q'}\subseteq Y_B,\ \mbox{and}\
  d_k(B_{p'},B_{q'})\le  16 \pi c c_m c_{m-4} k^{-m+6}.$$

Recall that $X_A\cap H_0\subseteq Y_A\cap H_0$ holds trivially and
 $X_B\cap H_0= Y_B\cap H_0$  was proved in  (4).
Hence we obtain
$$\begin{array}{clcl}
&d_k({\rm conv}(X_A\cap H_0),{\rm conv}(X_B\cap H_0))& \le &
d_k({\rm conv}(Y_A\cap
H_0),{\rm conv}(Y_B\cap H_0))\\[2ex]
\le&  d_k({\rm conv}(B_{p'}\cap H_0),{\rm conv}(B_{q'}\cap
H_0))&\stackrel{\ref{3.20}}{\le} &2d_k( B_{p'} ,B_{q'} )\\ [2ex]
\le&32\pi c c_m c_{m-4} k^{-m+6},&\end{array}$$ which had to be
shown.
\end{proof}

\begin{Corollary}\label{nucVG}
Let $G$ be a locally quasi--convex group. The following assertions
are equivalent:
\begin{itemize}
\item[(i)] $G$ is a nuclear group.
\item[(ii)] $\R^{G^\wedge}$ with the topology induced by the neighborhood
basis $(Y_U)$ (where $U$ runs through all quasi--convex
neighborhoods of $e\in G$) is a nuclear vector group.
\item[(iii)] $V_0$ with the topology induced by the neighborhood
basis $({\rm conv}(X_U\cap H_0))$ (where $U$ runs through all
quasi--convex neighborhoods of $e\in G$) is a nuclear vector group.
 \end{itemize}
\end{Corollary}

\begin{proof}
\ref{general} implies  $(i)\Rightarrow (ii)$ and $(i)\Rightarrow
(iii)$.

 Conversely, for any quasi--convex neighborhood $U$ we have
 $$ \phi_0(H_0\cap Y_U)\stackrel{(4)}{=}\phi_0(H_0\cap X_U)=U.$$
 Since $X_U\cap H_0\subseteq {\rm conv}(X_U\cap H_0)\cap H_0\subseteq Y_U\cap H_0
 \stackrel{(4)}{=}X_U\cap H_0 $, we obtain that both vector group topologies
 induce the same
topology on $H_0$ w.r.t. which $\phi_0$ is continuous and open.
Since the class of nuclear groups is closed w.r.t. forming subgroups
and Hausdorff quotients ((7.2) in \cite{Buch}), $G$ is nuclear and
hence $(ii)\Rightarrow (i)$ and $(iii)\Rightarrow (i)$.
\end{proof}

\begin{Lemma}\label{Lemma}
Let $(x_n)$ be a null sequence in $V_0$ contained in $H_0\cap X_U$
where $U$ is  a quasi--convex neighborhood of the neutral element
of $G$. The subspace $N:=\langle \{0\}^{U^\triangleright}\times
\Z^{G^\wedge\setminus U^\triangleright}\rangle$ is contained in
${\rm conv}(H_0\cap X_U)$ and

 $${\rm conv}\{\pm x_n|\  n\in\N\}\subseteq N+ {\rm conv}(H_0\cap X_A)$$
 where $A:= \{\pm \phi_0(x_n)|\ n\in\N\}$.
\end{Lemma}

\begin{proof}
It is obvious, that $\{0\}^{U^\triangleright}\times
\Z^{G^\wedge\setminus U^\triangleright}$ is a subset of $H_0\cap
X_U$. Hence $N\subseteq{\rm conv}(H_0\cap X_U)$.

Without loss of generality we may assume that $\{x_n|\
n\in\N\}=\{\pm x_n|\ n\in\N\}$.
 Fix any $(x_\chi)=x:=x_n\in H_0\cap
X_U$. This means that there exists $g\in G$  such that $|x_\chi|\le
1/4$ for all $\chi\in U^\triangleright$ and $e^{2\pi i
x_\chi}=\chi(g)$ for all $\chi \in
  G^\wedge$. In particular, for $\chi\in U^\triangleright,$
we have $\chi(g)\in\T_+$, which implies $g\in U$.

 For those $\chi\in A^\triangleright \setminus
U^\triangleright$ with $|x_\chi|>1/4$, we choose $k_\chi\in\Z$ such
that $|x_\chi-k_\chi|\le 1/4$. In all other coordinates we define
$k_\chi:=0$. Then $(x_\chi)-(k_\chi)\in H_0\cap X_A$, since
$g=\phi_0(x)=\phi_0((x_\chi-k_\chi))\in A$. The assertion follows,
since the right hand side is convex.

\end{proof}

\begin{Theorem}\label{anstrengend}
Let $G$ be a metrizable, locally quasi--convex Hausdorff group which
is not nuclear. Then there is a null sequence $(g_n)$ in $G$ and a
quasi--convex neighborhood $U$ of the neutral element such that for
no $n_0\in\N$ the  estimate
$$(d_k( \{\pm g_n:\ n\ge n_0
 \}\cup\{0\},U))\le(c k^{-9}) $$ holds if
 $c<\frac{\sqrt[4]{945}}{32\pi^4c_5c_9}$.
\end{Theorem}

\begin{proof}
Let $(U_n)$ be a decreasing neighborhood basis of the neutral
element of $G$ consisting of quasi--convex sets. Let $V_0,\ H_0$ and
$\phi_0$ be as introduced above. Since $G$ is not a nuclear group,
$V_0$ is not a nuclear vector group either (\ref{nucVG}).
 According to \ref{kompakt}, there exists a null sequence $(y_n)$ in
 $V_0$ and  a symmetric and convex
neighborhood of $0$ (we may assume that it is of the form ${\rm
conv}(H_0\cap X_U)$ for some a quasi--convex neighborhood $U$ of
the neutral element in $G$)
 such that for every $m\in\N$
$$(k^3d_k({\rm conv}\{\pm y_n|\ n\ge m\},{\rm conv}(H_0\cap X_U)))
\notin\ell^\infty$$ holds.

We choose $n_0$ such that $y_n\in {\rm conv}(H_0\cap X_{U})$ for all
$n\ge n_0$. Since every $y\in {\rm conv}(H_0\cap X_{U_k})$ belongs
to ${\rm conv}(F_k)$ for a suitable finite set $F_k\subseteq H_0\cap
X_{U_k}$, there is a null sequence $(x_k)$ in $ H_0$ such that
$\{\pm y_n|\ n\ge n_0\}\subseteq {\rm conv}\{\pm x_k|\ k\ge 1\}$ and
such that for every $m\in\N$ there exists $n_m\in\N$ with
\begin{eqnarray}\{\pm y_n|\ n\ge n_m\}\subseteq {\rm conv}\{\pm x_n|\ n\ge
m\}.\end{eqnarray}

  In
particular, \begin{eqnarray}(k^3d_k({\rm conv}\{\pm x_n:\ n\ge
m\},{\rm conv}(H_0\cap U)))\notin\ell^\infty\quad\forall\
m\in\N.\end{eqnarray}

Let $g_n:=\phi_0(x_n)$ and $A_{m}:=\{\pm g_n|\ n\ge m\}\cup\{0\}$.
According to \ref{Lemma},
\begin{eqnarray}{\rm conv}\{\pm x_n|\ n\ge
m\}\subseteq N+ {\rm conv}(H_0\cap X_{A_m}) \end{eqnarray} and
$N\subseteq {\rm conv}(X_U\cap H_0)$ for  $N=\langle
\{0\}^{G^\wedge\setminus U^\triangleright}\times
\Z^{U^\triangleright}\rangle$. Hence we obtain:
 \begin{eqnarray}
d_k({\rm conv}(X_{A_m}\cap H_0),{\rm conv}(X_U \cap
H_0))&\stackrel{\ref{properties}(iii)}{=} \nonumber\\ [1ex] d_k({\rm
conv}(X_{A_m}\cap H_0)+N,{\rm conv}(X_U\cap H_0))& \stackrel{(17)}{\ge} \nonumber\\
[1ex] d_k({\rm conv}(\{\pm x_n|\ n\ge m\}), {\rm conv}(X_U\cap
H_0))&
\stackrel{(15)}{\ge}\nonumber\\
[1ex] d_k({\rm conv}(\{\pm y_n|\ n\ge n_m\}), {\rm conv}(X_U\cap
H_0)).
\end{eqnarray}

Assume now that for some $m\in\N$ $$(d_k(A_m,U))\le (c k^{-9})
\quad\mbox{ where}\ \ c<\frac{\sqrt[4]{945}}{32\pi^4c_5c_9}. $$
\ref{general} implies that \begin{eqnarray}d_k( {\rm conv}(H_0\cap
X_{ A_m }),{\rm conv}(H_0\cap X_U))\le 32\pi c c_5c_9k^{-3},
\end{eqnarray}

since $ (16\pi cc_5c_9)^2\underbrace{\sum_{k\in\N}k^{-6}}_{\frac{
\pi^6 }{\sqrt{945}}}<1/4.$

Combining (18) and (19), we  obtain that $$(k^3 d_k({\rm conv}
\{y_n|\ n\ge n_m\}, {\rm conv}(X_U\cap H_0)))$$ is bounded. This
contradiction completes the proof.
\end{proof}

\section{Strongly reflexive groups}
 In this section we prove first that the dual space of a locally
 convex nuclear $k_\omega$--space is again nuclear. Afterwards we
 establish a technical lemma in order to prove the analogue for the
 group case. Combining these results with well known properties of
 nuclear groups and $k_\omega$--groups we obtain that nuclear
 $k_\omega$--groups are strongly reflexive.

\begin{Definition}
A Hausdorff topological space $X$ is called {\bf $k_\omega$--space},
if it has a countable cobasis for the compact sets and the topology
is the final topology induced by the compact subsets.

A Hausdorff space $(X,{\cal O})$ is called a {\bf $k$--space} if
${\cal O}$ coincides with the final topology induced by all compact
subsets of $X$.
\end{Definition}

Observe that every $k_\omega$--space is a $k$--space.

\begin{Proposition}\label{dual}
If $G$ is an abelian metrizable group then its character group is
complete and a $k_\omega$--space.

If $G$ is a topological group and a $k_\omega$--space, then its dual
group is complete and metrizable.
\end{Proposition}

\begin{proof}
The first assertion is a consequence of  (4.7) in \cite{Diss} or
Theorem 1 in \cite{MJ} and (1.11) in \cite{Buch}.

The second assertion is (2.8) in \cite{Diss} and (1.11) in
\cite{Buch}.

\end{proof}

Before stating the next proposition, we recall the following

\begin{Notation}
For an abelian topological group $(G,\tau)$ we denote by $G^\wedge$ the group of all
continuous characters $\chi:G\to \T$. Endowed with the compact--open topology,
$G^\wedge$ is an abelian Hausdorff group.This allows to define $G^{\wedge\wedge}$.
Let $$\alpha_G:G\longrightarrow G^{\wedge\wedge},\ x\mapsto \alpha_G(x):\chi\mapsto \chi(x)$$
denote the canonical homomorphism which in general is neither injective, nor continuous, nor surjective,  nor open.
\end{Notation}

\begin{Proposition}\label{alphaGcont}
If $G$ is a $k$--space then $\alpha_G$ is continuous. (This is
equivalent to: every compact subset of $G^\wedge$ is
equicontinuous.)
\end{Proposition}

\begin{proof}
This is a consequence of (1.1) and (2.3) in \cite{Noble}
\end{proof}

\begin{Lemma}\label{polar}
Let $A$ and $B$ be symmetric and convex subsets of a  vector space
$E$, Assume  that $B$ is absorbing. Then
$$d_k(B^\triangleright,A^\triangleright)\le kd_k(A,B)$$ where the
polars are formed in the algebraic dual $E^*$ of $E$.
\end{Lemma}

\begin{proof}
Let $A\subseteq c B +L$ where $L$ is an at most $(k-1)$--dimensional
subspace of $E$. Then $B^\triangleright\cap L^\bot \subseteq
cA^\triangleright$. Consider the  vector  space $E^*_B:=\{\phi\in
E^*|\ \phi(B)\ \mbox{is bounded}\}\supseteq B^\triangleright$; the
Minkowski functional $p_{B^\triangleright}$ is a norm, since $B$ was
assumed to be absorbing.
 By Auerbach's Lemma, there exists a projection
$\pi:E^*_B\rightarrow L^\bot\cap E^*_B$ with $
\pi(B^\triangleright)\subseteq k B^\triangleright \cap L^\bot$. The
assertion follows from $B^\triangleright\subseteq {\rm ker}\pi +
\pi(B^\triangleright)\subseteq    {\rm ker}\pi +
k(B^\triangleright\cap L^\bot)$ and the fact that ${\rm dim}\,( {\rm
ker}\pi) ={\rm dim} L\le k-1$.
\end{proof}

\begin{Theorem}\label{nukhkvr}
Let $E$ be a locally convex   vector space which is (as a
topological space) a   $k_\omega$--space. Then $E'_{co}$, the dual
space of $E$ endowed with the compact--open topology, is a nuclear
space if and only if $E$ is nuclear.
\end{Theorem}

\begin{proof}
  Since $E$ is a $k_\omega$--space, $E'_{co}$ is complete and metrizable
  (\ref{dual} and \ref{Smith}).  Hence $E'_{co}$ is a Fr\'{e}chet space.

Let us suppose first that $E$ is a nuclear locally convex vector
space and let us assume that $E'_{co}$ is not nuclear. According to
\ref{CharFrechet}, there exists a symmetric and convex neighborhood
$U$ of $0$ in $E'_{co}$ and a totally bounded convex symmetric
subset $K$ of $E'_{co}$ such that $(k^3d_k(K,U))$ is unbounded.
Replacing $K$ by its closure, we may assume that $K$ is compact.

According to \ref{alphaGcont},   $K$ is equicontinuous, which means
that $B:=K^{\triangleleft}=\{x\in E:\ |\phi(x)|\le 1\ \forall
\phi\in K\}$ is a neighborhood of $0$ in $E$. Since $E$ was assumed
to be nuclear, there exists another symmetric and convex
$0$--neighborhood $A$ such that $d_k(A,B)\le k^{-4}$ for all
$k\in\N$. Now \ref{polar} implies that
$d_k(B^\triangleright,A^\triangleright)\le k^{-3}$ holds for all
$k\in\N$.

Observe the $K\subseteq B^\triangleright$. Since $A^\triangleright$
is compact and hence bounded, there exists $c>0$ such that
$A^\triangleright\subseteq cU$. So we obtain:
$$
d_k(K,U)\stackrel{\ref{properties}(iii)}{=}c d_k(K,cU)\le
cd_k(B^\triangleright, A^\triangleright)\le ck^{-3},$$ which implies
that the sequence $(k^3d_k(K,U))_{k\in\N}$ is bounded in
contradiction to our assumption.

 Conversely, let us assume that
$E'_{co}$ is nuclear. According to \ref{20.36}, $(E'_{co})'_{co}$ is
a locally convex  nuclear vector space.
  Since the evaluation mapping $E\rightarrow
(E'_{co})'_{co},\ x\mapsto (f\mapsto f(x))$ is an embedding, $E$ is
a nuclear space as well.

\end{proof}

In order to prove an analogous result for groups, we need one more
technical lemma, which allows us to manipulate the factor $c$ in the
setting $(d_k(A,B))\le (ck^{-m})$.

\begin{Lemma}\label{Faktor}
Let $(g_n)$ be a null--sequence in a complete Hausdorff topological
group $G$ and let $W$ be a closed neighborhood of $0$ such that
$(d_k(\{\pm g_n:\ n\in\N\},W))\le (ck^{-m})$ for some $m\ge 1$. For
every $c_1>0$ there exists $n_0\in \N$ such that $$(d_k(\{\pm g_n:\
n\ge n_0\},W))\le (c_1k^{-m+1}).$$
\end{Lemma}

\begin{proof}
By assumption, there exists a vector space $E$, a subgroup $H$ of
$E$ and a homomorphism $\phi:H\rightarrow G$ and symmetric and
convex subsets $X,Y$ of $E$ which satisfy $d_k(X,Y)\le ck^{-m}$ for
all $k\in\N$ and
 \begin{eqnarray} \{\pm g_n:\
n\in\N\}\subseteq  \phi(X\cap H)\ \mbox{and}\ \  \phi(Y\cap
H)\subseteq W.\end{eqnarray}

Without loss of generality, we may assume that $E=\langle Y\rangle$.
Let $ \wt{E}$ be the completion of the seminormed space $(E,p_Y)$
(where $p_Y$ is the Minkowski functional of $Y$) and observe that
(20) implies that $\phi:H\rightarrow G$ is continuous w.r.t. the
induced topology. According to (10.19) in \cite{HR}, $\phi$ can be
extended to a continuous homomorphism $\ol{\phi}:\ol{H}\rightarrow
G$ (since $G$ is complete). Hence there exists $\eta>0$ such that
$\ol{\phi}(\ol{H}\cap \eta \ol{Y})\subseteq W$. Observe further that
$X\subseteq cY+L$ (for a subspace $L\le E$) implies
$$\ol{X}\subseteq X+\eps\ol{Y}\subseteq (c+\eps)\ol{Y}+L\quad\quad \forall\;\eps>0$$
and hence $d_k(\ol{X},\ol{Y})\le d_k(X,Y)$.

Hence we may assume that $E$ is complete, $X,\ Y$, and $H$ are
closed in $E$; nevertheless,  have to pay for this improvement by a (possibly)
larger constant, i.e. $\dis d_k(X,Y)\le \frac{c}{\eta}k^{-m}$
as replaced $Y$ by $\eta \ol{Y}$ ($X$ by $\ol{X}$ and $H$ by
$\ol{H}$).

 Let $K:={\rm ker}\,\phi$.
According to (20), it is possible to fix for every $n\in \N$ an
element $h_n\in X\cap H$ with $\phi(h_n)=g_n$.
 Let $A$ denote the set of accumulation points
of the sequence $(h_n)_{n\in\N}$. Since $H\cap X$ is closed, we have
$A\subseteq H\cap X$. We want to show that $A\subseteq K\cap X$. It
remains to show that $A\subseteq K$. For that we fix $a\in A$. There
exists a subsequence $(h_{n_k})$ of $(h_n)$ converging to $a$. Since
$\phi(h_{n_k})=g_{n_k}\rightarrow 0$ and (by the continuity of
$\phi$) $\phi(h_{n_k})\rightarrow \phi(a)$, we obtain $\phi(a)=0$
since $G$ is a Hausdorff group. This means $a\in K$.

We fix $k_0\in\N$ such that $\dis \frac{2c}{\eta k_0}\le c_1$ and
define $\eps:=
 c_1k_0^{-m+1}$.

Since $(d_k(X,Y))$ is a null--sequence, $X$ is a totally bounded
subset of $(E,p_Y)$ (cf. (9.1.4) in \cite{Pietschnuc}).  In
particular, all but a finite number of elements of the sequence
$(h_n)_{n\in\N}$ are contained in $A+\eps Y$. Indeed, assume  there
exist infinitely many members of the sequence $(h_n)$ outside
$A+\eps Y$. This infinite set must have an accumulation point $b$.
Of course, $b\in A$ and the neighborhood $b+\eps Y$ contains an
infinite number of them in contradiction to our assumption.

 We fix $n_0$ such that $h_n\in A+\eps
Y$ for all $n\ge n_0$.

For every $n\ge n_0$ we define   $h'_n:=h_n-a_n$ for suitable
$a_n\in A$ such that $h'_n\in \eps Y$. Hence $\{h'_n|\ n\ge
n_0\}\subseteq (X\cap H) - A\subseteq 2X$ and therefore  $$X':={\rm
conv}\{\pm h'_n:\ n\ge n_0\}\subseteq 2X \cap \eps Y.$$ For $k\ge
k_0$ we obtain
$$d_k(X',  Y)\stackrel{\ref{properties}(iii)}{\le}
2d_k(X,Y)\le  \frac{2c}{\eta  } k^{-m }\le\frac{2c}{\eta k_0}
k^{-m+1}\le c_1k^{-m+1}$$ and for $k<k_0:$ $$ d_k(X',Y)\le
d_1(X',Y)\le \eps= c_1k_0^{-m+1}\le c_1k^{-m+1}.$$

Now we have: $\{\pm g_n|\ n\ge n_0\}\subseteq \phi(X'\cap H),\ \phi(
Y\cap H)\subseteq W$ and $d_k(X', Y)\le c_1k^{-m+1}$, which
completes the proof.

\end{proof}

\begin{Proposition}\label{nuctb}
Let $G$ be a nuclear group. For every totally bounded subset
$S\subseteq G$ and every neighborhood $U$ of the neutral element
and all $n\in\N$, there exists a constant $c>0$ such that $(
d_k(S,U))\le (ck^{-n})$.
\end{Proposition}

\begin{proof}
Since $G$ is a nuclear group, there exists   neighborhood $W$ of
$0$ such that $(d_k(W,U))\le (k^{-n})$. By definition, this means,
there exists a vector space $E$ and symmetric and convex subsets $X$
and $Y$ such that $d_k(X,Y)\le k^{-n}$ for all $k\in \N$ and further
a subgroup $H$ of $E$ and a surjective group homomorphism,
$\phi:H\rightarrow G$ such that $W\subseteq\phi(X\cap H)$ and
$\phi(Y\cap U)\subseteq U$.

Since $S$ is totally bounded and $\phi$ is surjective, there exists
a finite set $F\subseteq H$ such that $$S\subseteq \phi(F) +
W\subseteq \phi(F + (X\cap H)) \subseteq \phi((X+F)\cap H).$$

If $X\subseteq c Y + L$ then $X+{\rm conv}F\subseteq c Y+ L+ \langle
F\rangle$, which implies $$d_{k+f}(X+{\rm conv}F,Y)\le d_k(X,Y)  $$
where $f:={\rm dim} \langle F\rangle$. In particular, $$(k+f)^nd_
{k+f}(X+{\rm conv}F,Y)\le k^n(1+f)^n d_k(X,Y)\le (1+f)^n .$$ Since
$S\subseteq \phi((F+X)\cap H)\subseteq \phi(({\rm conv}F + X)\cap
H)$, the assertion follows.

\end{proof}

\begin{Lemma}\label{Abschaetzung}
For subsets  $A$ and $B$ of an abelian topological group $G$ for
which $(d_k(A,B))\le (c k^{-n})$ holds for some $n\ge 5$,  the
polars satisfy $(d_k(B^\triangleright,A^\triangleright))\le (cc_{n}
k^{-n+5})$.
\end{Lemma}

\begin{proof}

This is (16.4) in \cite{Buch}.  
\end{proof}

\begin{Corollary}\label{alphaGdual}
Let $G$ be a nuclear group such that $\alpha_G$ is continuous.

For every compact subset $K\subseteq G^\wedge$ and every
neighborhood $W$ of the neutral element of $G^\wedge$ and every
$n\in\N$, there exists $c>0$ such that
$$( d_k(K,U))\le (ck^{-n}).$$
\end{Corollary}

\begin{proof}
By the definition of the compact--open topology, there exists a
compact subset $S\subseteq G$ such that $S^\triangleright\subseteq
U$. Since  $\alpha_G$ is continuous, $K$ is equicontinuous, which
means that $K\subseteq W^\triangleright$ for a  suitable
neighborhood $W$ of the neutral element of $G$. According to
\ref{nuctb}, we have $( d_k(S,W))\le (ck^{n+5})$ for a suitable
constant $c>0$. \ref{Abschaetzung} implies $(d_k(W^\triangleright,
S^\triangleright))\le (cc_{n+5} k^{-n})$ and hence $( d_k(K,U))\le (
d_k(W^\triangleright, S^\triangleright))\le (c c_{n+5}k^{-n})$,
which completes the proof.
\end{proof}

Now we a prepared to prove the main theorem of this article:

\begin{Theorem}\label{MainTheorem}
Let $G$ be  a nuclear $k_\omega$--group. Then $G^\wedge$ is a
completely metrizable nuclear group.
\end{Theorem}

\begin{proof}
Since $G$ is a $k_\omega$--group, $\alpha_G$ is continuous and
$G^\wedge$ is completely metrizable (\ref{dual}). Let us assume that
$G^\wedge$ is not nuclear. According to \ref{anstrengend}, there
exists a null sequence $(g_n)$ and a quasi--convex (and hence
closed) neighborhood $U$ such that $$(d_k(\{\pm g_n|\ n\ge
n_0\}\cup\{0\},U)\le (ck^{-9})$$ does not hold for any $n_0\in \N$
if $\dis c<\frac{\sqrt[4]{945}}{32\pi^4c_5c_9}=:c_0$. Since $\{\pm
g_n|\ n\ge n_0\}\cup\{0\}$ is compact and $U$ is a neighborhood of
$0$,
 there exists by \ref{alphaGdual} a constant $c_1>0$ such that
$$(d_k(\{\pm g_n|\ n\ge  0\}\cup\{0\},U)\le (c_1k^{-10}).$$
\ref{Faktor} implies the existence of $n_0\in \N$ such that
$$(d_k(\{\pm g_n|\ n\ge n_0\}\cup\{0\},U))\le (c_2k^{-9}) $$ where
$c_2=c_0/2$. This contradiction completes the proof.
\end{proof}

\begin{Theorem}\label{nucsr}
Every nuclear $k_\omega$--group  is strongly reflexive.
\end{Theorem}

\begin{proof}
Let $G$ be a nuclear $k_\omega$--group.  According to \cite{Raikov},
every $k_\omega$--group is complete. (21.5) in \cite{Diss} implies
that $\alpha_G$ is an open isomorphism. $G$ being a $k$--space
implies that  $\alpha_G$ is continuous  (\ref{alphaGcont}) and hence
$\alpha_G$ must be a topological isomorphism which means that $G$ is
reflexive.

 Every closed subgroup and every Hausdorff quotient group
of a $k_\omega$--group, resp. nuclear group is a $k_\omega$--group,
resp. nuclear. In particular, all closed subgroups and Hausdorff
quotient groups of $G$ are reflexive.

According to \ref{MainTheorem}, $G^\wedge$ is  a completely
metrizable nuclear group and hence, due to (17.3) in \cite{Buch},
$G^\wedge$ is strongly reflexive. Therefore, all closed subgroups and
Hausdorff quotient groups of $G^\wedge$ are reflexive as well.
\end{proof}

\begin{Definition}
An abelian topological group $G$ is called {\bf almost metrizable  }
if there is a compact subgroup $H$ such that $G/H$ is metrizable.
\end{Definition}

\begin{Examples} Every metrizable and every compact abelian group
is almost metrizable. Further, it is a consequence of the structure
theorem for locally compact abelian groups that every locally
compact abelian group is almost metrizable.
\end{Examples}

\begin{Definition}
A Hausdorff space  $X$ is called {\bf locally $k_\omega$ space} (when
 endowed with the subspace topology) if
every point has an open neighborhood which is  a $k_\omega$--space.
\end{Definition}

Let us first collect some properties of abelian groups which are
(locally) $k_\omega$ spaces. Following \cite{Helge}, we  will call them  (locally) $k_\omega$
groups.

\begin{Proposition}\label{Helge}
An abelian Hausdorff group  is a locally $k_\omega$ group if and
only if it has an open $k_\omega$ subgroup.

The dual group of a locally $k_\omega$  group is \v Cech complete
and conversely, the dual group of a \v Cech complete group is
locally $k_\omega$.

\end{Proposition}

\begin{proof}
These are (5.3) and (6.1) in \cite{Helge}.
\end{proof}

\begin{Theorem}\label{sr}
Every nuclear locally $k_\omega$ group is strongly reflexive.
\end{Theorem}

\begin{proof}
Let $G$ be a locally $k_\omega$ group. According to \ref{Helge},
there is an open subgroup $H$ of $G$ which is a $k_\omega$ group and
(trivially) nuclear. So \ref{nucsr} implies that $H$ is strongly
reflexive. In \cite{3Pers} it has been shown that a group is
strongly reflexive if it has an open, strongly reflexive subgroup.
\end{proof}

\begin{Remark}
 A result similar to \ref{sr} has been
stated in (4.6) in \cite{SMJ}, unfortunately the proof contains an
error: the homomorphism $(f^m_n)´:Q_n\rightarrow Q_m$ need not be an
embedding.

\ref{sr} generalizes (7.9) in \cite{Helge}.
\end{Remark}

\begin{Questions}
Is every strongly reflexive group a nuclear group?

Is every strongly reflexive group a $k$--space?

\end{Questions}


\begin{thebibliography}{10}

\bibitem{SMJ}Ardanza--Trevijano, S. and  Chasco, M.J., {\sl The Pontryagin
duality of sequential limits of topological Abelian grousp}, Journal
of pure and applied Algebra {\bf 202} (2005), 11-21.

\bibitem{Diss}Au\ss{}enhofer, L., {\sl Contributions to the duality
theory of abelian topological groups and to the theory of nuclear
groups}, Diss. Math. CCCLXXXIV, 1999.


\bibitem{W+M}Banaszczyk, M. and  Banaszczyk, W., {\sl
Characterization of nuclear spaces by means of additive subgroups}
  Math. Z., {\bf  186}  (1984), 125-133.

\bibitem{Buch}Banaszczyk, W., {\sl Additive Subgroups of Topological
Vector Spaces}, Lecture Notes 1466, Springer Verlag, Berlin,
Heidelberg, 1991.


\bibitem{3Pers}Banaszczyk, W., Chasco, M.J., and Mart\'{\i}n-Peinador,
E., {\sl Open subgroups and Pontryagin duality}, Math. Z. {\bf 215
(2)} (1994),195-204.

\bibitem{Bauhardt}Bauhardt, W., {\sl Hilbert--Zahlen von Operatoren
in Banachr\"aumen}, Math. Nachr. {\bf 79} (1977), 181 -- 186.

\bibitem{MJ}Chasco, M.J, {\sl
Pontryagin duality for metrizable groups}, Arch. Math. {\bf 70} No.1
(1998), 22-28.

\bibitem{Helge}Gl\"ockner, H.,  Gramlich, R. and Hartnick, nT., {\sl Final Group
Topologies, Kac--Moody groups and Pontryagin Duality}, Israel J. Math. {\bf 177}, (2010), 49 -- 102.

\bibitem{Grothendieck}Grothendieck, A., {\sl
Sur une notion de produit tensoriel topologique d'espaces vectoriels
topologiques, et une classe remarquable d'espaces vectoriels liée à
cette notion} (French), C. R. Acad. Sci., Paris, {\bf 233} (1951),
1556-1558.

\bibitem{HR}Hewitt, K.A. and Ross, {\sl Abstract Harmonic
Analysis}, Springer Verlag, Berlin, G\"ottingen, 1963.

\bibitem{Kenderov}Kenderov, P., {\sl On topological vector groups}, Math. USSR, Sb.
10 (1970), 531-546.

\bibitem{Noble}Noble, N., {\sl
k-groups and duality}, Trans. Am. Math. Soc. {\bf 151} (1970),
551-561.

\bibitem{Pietsch}Pietsch, A., {\sl Operator Ideals}, North--Holland
Publisching Company, Amsterdam, 1980.

\bibitem{Pietschnuc}Pietsch, A., {\sl Locally Convex Nuclear
Spaces}, Springer Verlag, Berlin, Heidelberg, $^2{1972}$.

\bibitem{Raikov}Ra\"{\i}kov,~D.\,A., {\sl On $B$-complete topological
vector groups} (Russian), Studia Math.  {\bf 31} (1968), 295--306.

\bibitem{RD}Roelcke, W. and  Dierolf, S., {\sl Uniform structures on topological
groups and their quotients},   McGraw-Hill International Book
Company, New York etc.,  1981.

\bibitem{Schaefer}Schaefer, H.H. and Wolff. M.P., {\sl Topological Vector spaces},
Springer Verlag, Heidelberg, Berlin, 1999.

\bibitem{Smith}Smith, M., {\sl
The Pontrjagin duality theorem in linear spaces}, Ann. Math. (2)
{\bf 56} (1952), 248-253.
\end{thebibliography}
\end{document}